\documentclass[12pt]{article}
\usepackage{amsfonts}
\usepackage{amssymb}
\usepackage{bbm}
\usepackage{amscd}
\usepackage{mathrsfs}
\usepackage{amsmath,amssymb, amsthm}
\usepackage{graphicx}
\usepackage{color}
\usepackage{authblk}
\usepackage{subcaption}
\usepackage{indentfirst,latexsym,amssymb}
\usepackage[bookmarks=false]{hyperref}
\usepackage{tikz}

\newtheorem{theorem}{Theorem}[section]

\newtheorem{conj}[theorem]{Conjecture}
\newtheorem{lem}[theorem]{Lemma}

\def\00{\mathbf{0}}

\def\Z{\mathbb{Z}}

\def\Aut{\hbox{\rm Aut}}

\def\Cay{\hbox{\rm Cay}}

\def\FF{{\mathcal F}}

\def\Ga{\Gamma}

\def\Proofs{\par\noindent{\bf Proof of Theorems~}}
\def\qed{\hfill$\Box$\vspace{12pt}}

\usepackage{xcolor}
\usepackage[normalem]{ulem}

\begin{document}

\title{Nowhere-zero 3-flows in Cayley graphs on supersolvable groups}
\author[a]{Junyang Zhang}
\author[b]{Sanming Zhou}
\affil[a]{\small School of Mathematical Sciences, Chongqing Normal University, Chongqing 401331, P. R. China}
\affil[b]{\small School of Mathematics and Statistics, The University of Melbourne, Parkville, VIC 3010, Australia}
\date{}

\openup 0.5\jot
\maketitle

\renewcommand{\thefootnote}{\fnsymbol{footnote}}
 \footnotetext{E-mail addresses: jyzhang@cqnu.edu.cn (Junyang Zhang); sanming@unimelb.edu.au (Sanming Zhou)}

\begin{abstract}
Tutte's 3-flow conjecture asserts that every $4$-edge-connected graph admits a nowhere-zero $3$-flow. We prove that this conjecture is true for every Cayley graph of valency at least four on any supersolvable group with a noncyclic Sylow $2$-subgroup and every Cayley graph of valency at least four on any group whose derived subgroup is of square-free order.


\end{abstract}

\section{Introduction}
\label{sec:int}

All graphs considered in this paper are finite, undirected and loopless, with parallel edges allowed. In other words, graphs considered in this paper are loopless multigraphs \cite{BM2008}. We use the standard term ``simple graph" to mean a graph with no parallel edges. As usual, for a graph $\Ga$, we use $V(\Ga)$ and $E(\Ga)$ to denote its vertex set and edge set, respectively. The reader is referred to \cite{BM2008} for graph theoretical terminoloy not defined in this paper.

An \emph{orientation} of a graph $\Gamma$ is a digraph $D$ with vertex set $V(\Gamma)$ which is obtained from $\Gamma$ by endowing each edge with one of the two possible directions, where parallel edges may be endowed with opposite directions. Each oriented edge thus obtained is called an \emph{arc} of $D$. For convenience, we use the same notation $D$ to denote the set of arcs of $D$. So whenever we write $e \in D$ we mean $e$ is an arc of $D$. For a vertex $v\in V(\Gamma)$, we use $D^{+}(v)$ to denote the set of arcs of $D$ with tail $v$ and $D^{-}(v)$ the set of arcs of $D$ with head $v$. Let $k$ be a positive integer. A \emph{$k$-flow} \cite{BM2008} in $\Gamma$ is a pair $(D, \varphi)$, where $D$ is an orientation of $\Ga$ and $\varphi$ is a map from $D$ to the set of integers such that $|\varphi(e)|<k$ for every $e \in D$ and
$$
\sum\limits_{e\in D^{+}(v)}\varphi(e) =\sum\limits_{e\in D^{-}(v)}\varphi(e)
$$
for every $v\in V(\Gamma)$. If, in addition, $\varphi(e) \ne 0$ for every $e \in D$, then $(D, \varphi)$ is called a \emph{nowhere-zero $k$-flow} in $\Gamma$. It is well known (see, for example, \cite[Theorem 1.2.8]{Z1997}) that whenever $\Gamma$ admits a nowhere-zero $k$-flow $(D, \varphi)$ for some orientation $D$ of $\Ga$, then it admits a nowhere-zero $k$-flow $(D', \varphi')$ for any orientation $D'$ of $\Gamma$. In other words, the existence of a nowhere-zero $k$-flow only depends on the graph but not on any specific choice of orientation. As such we can choose any orientation when determining whether a graph admits a nowhere-zero $k$-flow. Obviously, if a graph admits a nowhere-zero $k$-flow, then it admits a nowhere-zero $(k+1)$-flow.

The study of nowhere-zero integer flows in graphs began with Tutte \cite{T1949,T1954} who proved that a planar graph admits a nowhere-zero $4$-flow if and only if the Four Color Conjecture holds. Up to now researchers have produced a large number of results on nowhere-zero integer flows, as one can find in the survey \cite{J1988} and monograph \cite{Z1997}. In \cite{T1949,T1954}, Tutte proposed three important conjectures on integer flows which are still open in their general form. One of them is the following well-known $3$-flow conjecture (see \cite[Conjecture 1.1.8]{Z1997}).

\begin{conj}[Tutte]
\label{Tutte}
Every $4$-edge-connected graph admits a nowhere-zero $3$-flow.
\end{conj}

This conjecture has been studied extensively for more than half a century. In 1979, Jaeger
\cite{J1979} proved that every $4$-edge-connected graph admits a nowhere-zero $4$-flow, and he further conjectured that there is a positive integer $k$ such that every $k$-edge-connected graph admits a nowhere-zero $3$-flow. Jaeger's conjecture was confirmed by Thomassen \cite{Th2012} who proved that the statement is true when $k=8$. This breakthrough was further improved by Lov\'asz, Thomassen, Wu and Zhang \cite{LTWZ2013} who proved that every $6$-edge-connected graph admits a nowhere-zero $3$-flow.

We follow \cite{R1995} for group theoretical terminology. All groups considered in this paper are finite. A graph is called \emph{vertex-transitive} if its automorphism group acts transitively on its vertex set. It is obvious that every vertex-transitive graph is regular. The following concept introduced by Authur Cayley enables us to construct many, but not all, vertex-transitive graphs: Let $G$ be a group with identity element $1$. A multiset $X$ with elements from $G\setminus\{1\}$ is called a \emph{connection multiset} if it satisfies $X = X^{-1} := \{x^{-1}: x \in X\}$ and for each $x \in X$ the multiplicities of $x$ and $x^{-1}$ in $X$ are equal. Given a connection multiset $X$ of $G$, the \emph{Cayley graph} on $G$ with respect to $X$, denoted by $\Cay(G,X)$, is defined to be the graph with vertex set $G$ such that the number of edges joining $g$ and $h$ is equal to the multiplicity of $g^{-1}h$ in $X$. (In particular, if $g^{-1}h \notin X$, then there is no edge between $g$ and $h$.) Note that $\Cay(G,X)$ may contain parallel edges but is loopless as $1 \notin X$. Obviously, $\Cay(G, X)$ is regular, with valency the cardinality of $X$, and it is connected if and only if $X$ generates $G$. In the case when $X$ is a \emph{connection set} (that is, the multiplicity of each element of $X$ is $1$), $\Cay(G,X)$ is a simple Cayley graph.

Nowhere-zero integer flows in vertex-transitive graphs have attracted considerable attention.
Alspach and Zhang conjectured (at the Louisville workshop on Hamilton cycles in 1992) that every Cayley graph with valency at least two admits a nowhere-zero $4$-flow. Since every $4$-edge-connected graph admits a nowhere-zero $4$-flow \cite{J1979}, this conjecture is reduced to the cubic case.
Alspach, Liu and Zhang \cite[Theorem 2.2]{ALZ1996} confirmed this conjecture for cubic simple
Cayley graphs on solvable groups. In \cite{NS2001}, Nedela and \v Skoviera proved that any counterexample to the conjecture of Alspach and Zhang must be a regular cover over a Cayley graph on an almost simple group. In \cite{P2004}, Poto\v cnik proved that if a connected cubic simple graph other than the Petersen graph admits a solvable vertex-transitive group of automorphisms, then it admits a nowhere-zero $4$-flow. This can be viewed as a generalization of the above-mentioned result of Alspach, Liu and Zhang \cite[Theorem 2.2]{ALZ1996} as the Petersen graph is not a Cayley graph.

It is well known that every connected vertex-transitive graph of valency $k\geq 1$ is $k$-edge connected \cite{M1971}. Thus, when restricted to the class of vertex-transitive graphs, Conjecture \ref{Tutte} asserts that every vertex-transitive graph of valency at least four admits a nowhere-zero 3-flow. In this regard, Poto\v cnik, \v Skoviera and \v Skrekovski \cite{PSS2005} proved that every Cayley graph of valency at least four on an abelian group admits a nowhere-zero $3$-flow. This result was improved by N\'an\'asiov\'a and \v Skoviera \cite[Theorem 4.3]{NS2009} who proved that Tutte's $3$-flow conjecture is true for Cayley graphs of valency at least four on nilpotent groups. Conjecture \ref{Tutte} has also been confirmed for simple Cayley graphs of valency at least four on dihedral groups \cite{YL2011}, generalized dihedral groups \cite{LL2015}, generalized quaternion groups \cite{LL2015}, generalized dicyclic groups \cite{AI2019}, and groups of order $pq^2$ for any pair of primes $p, q$ \cite{ZZ2022}. In \cite{LZ2016}, Li and Zhou proved that Conjecture \ref{Tutte} is true for simple graphs with valency at least four which admit a solvable arc-transitive group of automorphisms. In \cite{ZT2022}, Zhang and Tao proved that Conjecture \ref{Tutte} is true for simple graphs with order twice an odd number and valency at least four whose automorphism groups contain a solvable vertex-transitive subgroup which contains a central involution. In a recent paper \cite{ZZ22}, the authors proved that Conjecture \ref{Tutte} is true for vertex-transitive graphs of valency at least four whose automorphism groups contain a nilpotently vertex-transitive subgroup.

In this paper we prove that Conjecture \ref{Tutte} is true for Cayley graphs on two families of supersolvable groups. Recall from group theory that a group $G$ is called \emph{solvable} \cite{R1995} if it possesses a normal series $\{1\} = G_{0}\leq G_{1} \leq\cdots \leq G_{n} =G$ such that the quotient group $G_{i}/G_{i-1}$ is abelian for $1 \leq i \leq n$. A group $G$ is said to be \emph{supersolvable} \cite{R1995} if it has such a normal series with the property that $G_{i}/G_{i-1}$ is cyclic for $1 \leq i \leq n$. The \emph{derived subgroup} of a group $G$, denoted by $G'$, is the subgroup generated by all commutators $[x,y]:=x^{-1}y^{-1}xy$ ($x, y\in G$) of $G$.

The main results in this paper are as follows.

\begin{theorem}
\label{Cayley}
Every connected Cayley graph of valency at least four on any supersolvable group with a noncyclic Sylow $2$-subgroup admits a nowhere-zero $3$-flow.
\end{theorem}

\begin{theorem}
\label{Cayley1}
Every Cayley graph of valency at least four on any group whose derived subgroup is of square-free order admits a nowhere-zero $3$-flow.
\end{theorem}

Since any group whose derived subgroup is of square-free order is supersolvabe (see Lemma~\ref{derived}), Theorems \ref{Cayley} and \ref{Cayley1} affirm Tutte's $3$-flow conjecture for Cayley graphs on two families of supersolvable groups.

The rest of this paper is organized as follows. The next section contains three basic lemmas. In Section \ref{sec:ladder}, we will prove a few results about nowhere-zero $3$-flows in what we call ``generalized closed ladders'', namely, circular ladders $\Cay(\Z_{n} \times \Z_2, \{(1,0), (-1,0), (0,1)\})$, M\"obius ladders $\Cay(\Z_{2n}, \{1, -1, n\})$, and graphs obtained from these ladders by subdividing their ``rail edges''. As will be seen later, these results will play a key role in the proof of our main results, and we expect that they may also be useful in studying nowhere-zero $3$-flows in some other graphs not considered in this paper. In Section \ref{sec:lem}, we will prove several technical lemmas on Cayley graphs, especially those which are circular or M\"obius ladders, and the existence of nowhere-zero $3$-flows in such graphs. With these preparations, we will finally present our proof of Theorems \ref{Cayley} and \ref{Cayley1} in Section \ref{sec:proof}.

\section{Basic lemmas}
\label{sec:lem}

The \emph{union} $\Gamma_{1} \cup \Gamma_{2}$ of two graphs $\Gamma_{1}$ and $\Gamma_{2}$ is the graph with vertex set $V(\Gamma_{1})\cup V(\Gamma_{2})$ and edge set $E(\Gamma_{1})\cup E(\Gamma_{2})$ (see \cite[p.2]{BM2008}). Since this operation is associative, we can talk about the union $\cup_{i=1}^n \Gamma_{i}$ of $n$ graphs $\Gamma_{1}, \ldots, \Gamma_{n}$ for any $n \ge 1$.

\begin{lem}
\label{union}
Let $\Gamma$ be a graph. Suppose that there are $n \ge 2$ subgraphs $\Gamma_{1}, \Gamma_{2}, \ldots, \Gamma_{n}$ of $\Gamma$ such that
\begin{enumerate}
\item $\Gamma = \cup_{i=1}^n \Gamma_{i}$;
\item for $1 \le i \le n$, $\Gamma_{i}$ admits a nowhere-zero 3-flow; and
\item for $2 \le l \le n$, $\cup_{i=1}^{l-1} \Gamma_{i}$ and $\Gamma_{l}$ have at most one edge in common.
\end{enumerate}
Then $\Gamma$ admits a nowhere-zero 3-flow.
\end{lem}

\begin{proof}
We proceed by induction on $n$. First, we assume $n=2$. Fix an orientation $D$ of $\Gamma$. Let $D_{i}$ be the restriction of $D$ to $\Gamma_{i}$, for $i=1,2$. Since $\Gamma_{i}$ admits a nowhere-zero 3-flow, there exists a nowhere-zero $3$-flow $(D_{i},\varphi_{i})$ of $\Gamma_{i}$, for $i=1,2$. Note that, by (iii), $\Gamma_{1}$ and $\Gamma_{2}$ have at most one common edge. Define a map $\varphi$ from $D$ to $\Z$ as follows: If $\Gamma_{1}$ and $\Gamma_{2}$ have no common edges, then set
$$
\varphi(e) = \begin{cases}
                \varphi_{1}(e), & \mbox{if}~ e\in D_{1}, \\
                \varphi_{2}(e), & \mbox{if}~ e\in D_{2}; \\
              \end{cases}
$$
if $\Gamma_{1}$ and $\Gamma_{2}$ have a common edge $e_{0}$ and $\varphi_{1}(e_{0})\equiv\varphi_{2}(e_{0})\pmod{3}$, then set
$$
 \varphi(e)= \begin{cases}
                \varphi_{1}(e_{0})+\varphi_{2}(e_{0}), &\mbox{if}~e=e_{0},\\
                \varphi_{1}(e), & \mbox{if}~ e\in D_{1}\setminus\{e_{0}\}, \\
                \varphi_{2}(e), & \mbox{if}~ e\in D_{2}\setminus\{e_{0}\}; \\
              \end{cases}
$$
if $\Gamma_{1}$ and $\Gamma_{2}$ have a common edge $e_{0}$ and $\varphi_{1}(e_{0})\equiv-\varphi_{2}(e_{0})\pmod{3}$, then set
$$
 \varphi(e)= \begin{cases}
                \varphi_{1}(e_{0})-\varphi_{2}(e_{0}), &\mbox{if}~e=e_{0},\\
                \varphi_{1}(e), & \mbox{if}~ e\in D_{1}\setminus\{e_{0}\}, \\
                -\varphi_{2}(e), & \mbox{if}~ e\in D_{2}\setminus\{e_{0}\}. \\
              \end{cases}
$$
It is straightforward to verify that in each case $(D,\varphi)$ is a nowhere-zero 3-flow in $\Gamma$.

Suppose inductively that for some $n>2$ the result holds when $\Ga$ is the union of $n-1$ subgraphs satisfying (i)--(iii). Let $\Ga = \cup_{i=1}^{n} \Gamma_{i}$ satisfy (i)--(iii). Then by our hypothesis $\cup_{i=1}^{n-1} \Gamma_{i}$ admits a nowhere-zero 3-flow. Since $\Gamma=(\cup_{i=1}^{n-1} \Gamma_{i}) \cup \Gamma_{n}$ and by (iii), $\cup_{i=1}^{n-1} \Gamma_{i}$ and $\Gamma_{n}$ have at most one edge in common, it follows from the base case above that $\Gamma$ admits a nowhere-zero 3-flow.
\end{proof}

Let $\Ga$ be a graph, and let $e$ be an edge of $\Ga$ with end-vertices $u$ and $v$. The \emph{subdivision} of $e$ yields the graph obtained from $\Ga$ by deleting $e$, adding a new vertex $w$, and then adding an edge joining $u$ and $w$ and another edge joining $w$ and $v$. A \emph{subdivision} of $\Ga$ is a graph obtained from $\Ga$ by recursively applying edge subdivisions. In other words, a subdivision of $\Ga$ is obtained from $\Ga$ by replacing each of its edges by a path of length at least one. The following result is widely known and can be easily verified.

\begin{lem}
\label{subdivision}
For any positive integer $k$, a graph admits a nowhere-zero $k$-flow if and only if any subdivision of it admits a nowhere-zero $k$-flow.
\end{lem}

It is known (see \cite[Theorem 1.2.10]{Z1997}) that a graph admits a nowhere-zero $2$-flow if and only if all its vertices have even valencies. Note that a nowhere-zero $2$-flow is also a nowhere-zero $3$-flow. It is also known (see \cite[Theorem 4.1.4]{Z1997}) that a $2$-edge-connected cubic graph admits a nowhere-zero $3$-flow if and only if it is bipartite. These results together with Lemma \ref{union} imply the following lemma.

\begin{lem}
\label{regular}
Let $\Gamma$ be a regular graph of odd valency. If $\Gamma$ has a regular spanning subgraph of odd valency which admits a nowhere-zero 3-flow, then $\Gamma$ admits a nowhere-zero 3-flow. In particular, $\Gamma$ admits a nowhere-zero 3-flow if it has a cubic bipartite spanning subgraph.
\end{lem}

\section{Generalized closed ladders}
\label{sec:ladder}

The \emph{ladder graph} $L_n$ (where $n \ge 1$) is the Cartesian product of a path of order $n$ with the complete graph $K_2$ of order $2$. Note that $L_n$ is a simple graph. The \emph{circular ladder} $CL_n$ with $n$ rungs (where $n \ge 2$) is the Cartesian product $C_n \Box K_2$, where $C_n$ is the cycle of length $n$, with $C_2$ considered as the graph of two parallel edges between two distinct vertices. Equivalently, $CL_n$ is the Cayley graph $\Cay(\Z_{n} \times \Z_2, \{(1,0), (-1,0), (0,1)\})$. The edges $\{(i,0), (i,1)\}$, $i \in \Z_{n}$ of this Cayley graph are called the \emph{rungs} of $CL_n$, and the other $2n$ edges are called the \emph{rail edges} of $CL_n$. Circular ladders are also known as \emph{prism graphs} in the literature. The \emph{M\"obius ladder} $M_n$ with $n$ rungs (where $n \ge 2$) is the circulant graph $\Cay(\Z_{2n}, \{1, -1, n\})$. The edges $\{i, i+n\}$, $i \in \Z_{2n}$ of $\Cay(\Z_{2n}, \{1, -1, n\})$ are called the \emph{rungs} of $M_n$, and the other $2n$ edges are called the \emph{rail edges} of $M_n$. A \emph{closed ladder} is a circular or M\"obius ladder. Note that all closed ladders except $CL_{2}$ are simple graphs. Note also that each of $CL_{n}$ and $M_{n}$ contains $L_{n}$ as a subgraph.

The following lemma is known in the literature and can be easily verified.

\begin{lem}
\label{ladder}
Let $n\geq2$ be an integer. Then the following hold:
\begin{enumerate}
  \item $L_{n}$ admits a nowhere-zero $3$-flow;
  \item $CL_{n}$ admits a nowhere-zero $3$-flow if and only if $n$ is even;
  \item $M_{n}$ admits a nowhere-zero $3$-flow if and only if $n$ is odd.
\end{enumerate}
\end{lem}

Let $n \ge 2$ be an integer. Any graph obtained from $CL_n$ by replacing each rail edge by
a path of length at least one is called a \emph{generalized circular ladder} with $n$ rungs,
and any graph obtained from $M_n$ by replacing each rail edge by a path of length at least one is called a \emph{generalized M\"obius ladder} with $n$ rungs. The edges on such paths replacing rail edges are again called rail edges. A \emph{generalized closed ladder} is a generalized circular or M\"obius ladder. Equivalently, we can define a generalized circular (respectively, M\"obius) ladder with $n$ rungs to be a graph obtained from $CL_n$ (respectively, $M_n$) by a series of subdivisions of rail edges. It is evident that all generalized closed ladders are simple graphs, except the ones obtained from $CL_2$ in such a way that at least one pair of parallel edges are not subdivided.

Given a generalized circular ladder with two rungs, by deleting both or one of its rungs we obtain the vertex-disjoint union $C_m \cup C_n$ of two cycles or the vertex-disjoint union $C_m \cup C_n$ with an additional edge joining a vertex of $C_m$ and a vertex of $C_n$, respectively, for some integers $m, n \ge 2$. Denote the latter graph by $C^*_{m,n}$. By abuse of notation, we treat $C_m \cup C_n$ and $C^*_{m,n}$ as generalized circular ladders with $0$ and $1$ rung, respectively. Similarly, for a generalized M\"obius ladder with two rungs, by deleting both or one of its rungs we obtain a graph which is isomorphic to $C_n$ or $C_n + e$ for some integer $n \ge 4$, respectively, where $e$ is an edge joining two non-consecutive vertices of $C_n$. We treat $C_n$ and $C_n + e$ as generalized M\"obius ladders with $0$ and $1$ rung, respectively.

It is easy to see that $C_m \cup C_n$, $C_n$ and $C_n + e$ all admit nowhere-zero $3$-flows but $C^*_{m,n}$ does not. Combining this with Lemmas \ref{subdivision} and \ref{ladder} and the definition of a generalized closed ladder, we obtain the following lemma.

\begin{lem}
\label{minus}
Let $\Ga$ be a generalized circular ladder with at least one rung or a generalized M\"obius ladder with at least two rungs, and let $e$ be a rung of $\Ga$. Then $\Ga$ admits a nowhere-zero $3$-flow if and only if $\Ga - e$ does not admit any nowhere-zero $3$-flow.
\end{lem}

The following result will play a key role in our proof of Theorems \ref{Cayley} and \ref{Cayley1}.

\begin{theorem}
\label{cup}
Let $\Gamma$ be a graph, and let $\Lambda$ and $\Sigma$ be subgraphs of $\Ga$ (possibly with $V(\Lambda) \ne V(\Sigma)$) such that $\Ga = \Lambda \cup \Sigma$. Suppose that $\Lambda$ admits a nowhere-zero $3$-flow. Suppose further that every connected component $\Theta$ of $\Sigma$ is a generalized closed ladder and satisfies the following conditions:
\begin{enumerate}
\item $\Theta$ admits a nowhere-zero $3$-flow if it has no common edges with $\Lambda$;
\item each common edge of $\Theta$ and $\Lambda$, whenever it exists, is a rung of $\Theta$.
\end{enumerate}
Then $\Gamma$ admits a nowhere-zero $3$-flow.
\end{theorem}

\begin{proof}
We proceed by induction on the number of connected components of $\Sigma$. The result is trivially true when the number of connected components of $\Sigma$ is $0$ (that is, when $\Sigma$ is the null graph). Suppose that for some positive integer $m$ the result holds when the number of connected components of $\Sigma$ is $m-1$. Now assume that the number of connected components of $\Sigma$ is $m$. We will show that $\Gamma$ admits a nowhere-zero $3$-flow and thus complete the proof by mathematical induction. Let $\Theta$ be a connected component of $\Sigma$ and set $\Gamma'=\Lambda\cup \Sigma'$, where $\Sigma'$ is the graph obtained from $\Sigma$ by deleting its component $\Theta$. Since $\Gamma=\Lambda\cup\Sigma$, we have $\Gamma=\Gamma'\cup \Theta$. By the induction hypothesis, $\Gamma'$ admits a nowhere-zero $3$-flow.

Consider first the case when $\Lambda$ and $\Theta$ have no common edges. In this case, $\Theta$ admits a nowhere-zero $3$-flow by condition (i), and $\Gamma'$ and $\Theta$ have no common edges. Since $\Gamma=\Gamma'\cup \Theta$ and both $\Gamma'$ and $\Theta$ admit nowhere-zero $3$-flows, it follows from Lemma \ref{union} that $\Gamma$ admits a nowhere-zero $3$-flow.

Now consider the case when $\Lambda$ and $\Theta$ have at least one common edge. By condition (ii), each common edge of them is a rung of $\Theta$. Moreover, since $\Theta$ is a connected component of $\Sigma$, each common edge of $\Lambda$ and $\Theta$ is a common edge of $\Gamma'$ and $\Theta$, and vice versa. Take a common edge $e$ of $\Lambda$ and $\Theta$. Then $e$ is a rung of $\Theta$ and is a common edge of $\Gamma'$ and $\Theta$. Let $\Theta'$ be the graph obtained from $\Theta$ by deleting all common edges of $\Gamma'$ and $\Theta$ except $e$. Since $\Theta$ is a generalized closed ladder and all common edges of $\Gamma'$ and $\Theta$ are rungs of $\Theta$, we see that $\Theta'$ is a generalized closed ladder and $e$ is a rung of $\Theta'$. If $e$ is the unique rung of $\Theta'$, then $\Theta' - e$ is $2$-regular and hence admits a nowhere-zero $3$-flow. If $\Theta'$ has at least two rungs, then by Lemma \ref{minus}, either $\Theta'$ or $\Theta' - e$ admits a nowhere-zero $3$-flow. Since $\Gamma=\Gamma'\cup\Theta=\Gamma'\cup\Theta'=\Gamma'\cup(\Theta' - e)$, it follows from Lemma \ref{union} that $\Gamma$ admits a nowhere-zero $3$-flow.
\end{proof}

The last result in this section, Theorem \ref{cladder} below, will also play an important role in our proof of Theorems \ref{Cayley} and \ref{Cayley1}. This result yields \cite[Lemma 3]{LL2015} in the special case when $\FF$ is a family of closed ladders and $E$ is the union of the sets of rungs of the closed ladders in $\FF$. However, our proof below is different from the proof of \cite[Lemma 3]{LL2015}.

\begin{theorem}
\label{cladder}
Let $\Gamma$ be a graph. Suppose that there exist a family $\mathcal{F}$ of subgraphs of $\Ga$ and a subset $E$ of $E(\Gamma)$ such that the following conditions are satisfied:
\begin{enumerate}
\item $\Gamma$ is the union of the subgraphs in $\mathcal{F}$;
\item each subgraph $\Sigma \in \mathcal{F}$ is a generalized closed ladder, with each edge in $E(\Sigma) \cap E$ a rung of $\Sigma$, and moreover $|E(\Sigma) \cap E| \ge 2$ if $\Sigma$ does not admit any nowhere-zero $3$-flow;
\item each edge in $E$ is a rung of at least two distinct subgraphs in $\mathcal{F}$, and each edge in  $E(\Gamma) \setminus E$ is contained in a unique subgraph in $\mathcal{F}$.
\end{enumerate}
Then $\Gamma$ admits a nowhere-zero $3$-flow.
\end{theorem}

\begin{proof}
We proceed by recursively constructing a sequence of subgraphs of $\Gamma$ each admitting a nowhere-zero $3$-flow such that the last subgraph in the sequence is $\Ga$ itself. For convenience, any edge in $E$ is called an $E$-rung and a graph is called valid or invalid depending on whether or not it admits a nowhere-zero $3$-flow. Denote by $n$ the cardinality of $\mathcal{F}$. Our construction goes as follows.

\begin{enumerate}
\item[(1)]
\begin{itemize}
\item[(a)] Choose an arbitrary $\Sigma_{1} \in \mathcal{F}$.

\item[(b)] If $\Sigma_{1}$ is valid, then we set $\Gamma_{1} = \Sigma_{1}$.

\item[(c)] If $\Sigma_{1}$ is invalid, then by condition (ii), $\Sigma_{1}$ contains at least two $E$-rungs. Take an $E$-rung $e_{1}$ contained in $\Sigma_{1}$ and set $\Gamma_{1} = \Sigma_{1} - e_{1}$. By Lemma \ref{minus}, $\Gamma_{1}$ is valid.

\item[(d)] Obviously, we have $\Gamma_{1} = \Sigma_{1}$ or $\Gamma_{1} = \Sigma_{1} - e_{1}$.
\end{itemize}

\item[(2)]
\begin{itemize}
\item[(a)] Select $\Sigma_{2} \in \mathcal{F}$ using the following rules: If $\Sigma_{1}$ is valid, then choose an arbitrary $\Sigma_{2}\in\mathcal{F}\setminus\{\Sigma_{1}\}$; if $\Sigma_{1}$ is invalid, then choose $\Sigma_{2}\in\mathcal{F}\setminus\{\Sigma_{1}\}$ such that $e_{1}$ is a common edge of $\Sigma_{1}$ and $\Sigma_{2}$. The existence of $\Sigma_{2}$ in the latter case is guaranteed by condition (iii). Note that in both cases we have $\Gamma_{1}\cup\Sigma_{2}=\Sigma_{1}\cup\Sigma_{2}$.

\item[(b)] If $\Gamma_{1}\cup\Sigma_{2}$ is valid, then we set $\Gamma_{2} = \Gamma_{1}\cup\Sigma_{2}$.
\item[(c)] If $\Gamma_{1}\cup\Sigma_{2}$ is invalid, then by Theorem \ref{cup}, $\Gamma_{1}$ and $\Sigma_{2}$ have no common edges and $\Sigma_{2}$ is invalid. By condition (ii), $\Sigma_{2}$ contains an $E$-rung, say, $e_{2}$, which is different from $e_1$. Since $\Gamma_{1}$ and $\Sigma_{2}$ have no common edges, by the definition of $\Ga_1$ we see that $e_2$ is not contained in $\Sigma_{1}$. Set $\Gamma_{2}=\Gamma_{1}\cup(\Sigma_{2} - e_{2})$. By Lemmas \ref{union} and \ref{minus}, $\Gamma_{2}$ is valid.

\item[(d)] Since $\Gamma_{1}\cup\Sigma_{2}=\Sigma_{1}\cup\Sigma_{2}$, we have $\Gamma_{2} = \Sigma_{1}\cup\Sigma_{2}$ or $\Gamma_{2} = (\Sigma_{1}\cup\Sigma_{2}) - e_{2}$.
\end{itemize}

\item[(3)] Inductively, for $i = 3, \ldots, n$, do the following:
\begin{itemize}
\item[] Assume that we have constructed a valid subgraph $\Gamma_{i-1}$ of $\Ga$ such that either $\Gamma_{i-1} = \Sigma_{1}\cup\cdots\cup\Sigma_{i-1}$ or $\Gamma_{i-1}=(\Sigma_{1}\cup\cdots\cup\Sigma_{i-1}) - e_{i-1}$, where $\Sigma_{1}, \ldots, \Sigma_{i-1}$ are pairwise distinct members of $\mathcal{F}$ and $e_{i-1}$ is an $E$-rung contained in $\Sigma_{i-1}$ but not in $\Sigma_{1}\cup\cdots\cup\Sigma_{i-2}$.

\item[(a)] Select $\Sigma_{i} \in \mathcal{F}$ using the following rules: If $\Gamma_{i-1}=\Sigma_{1}\cup\cdots\cup\Sigma_{i-1}$, then choose an arbitrary $\Sigma_{i}\in\mathcal{F}\setminus\{\Sigma_{1},\ldots,\Sigma_{i-1}\}$; if $\Gamma_{i-1}=(\Sigma_{1}\cup\cdots\cup\Sigma_{i-1}) - e_{i-1}$, then choose $\Sigma_{i}\in\mathcal{F}\setminus\{\Sigma_{1},\ldots,\Sigma_{i-1}\}$ such that $e_{i-1}$ is a common edge of $\Sigma_{i-1}$ and $\Sigma_{i}$. The existence of $\Sigma_{i}$ in the latter case is guaranteed by condition (iii) and the fact that $e_{i-1}$ is not in $\Sigma_{1}\cup\cdots\cup\Sigma_{i-2}$.

\item[(b)] If $\Gamma_{i-1}\cup\Sigma_{i}$ is valid, then we set $\Gamma_{i}=\Gamma_{i-1}\cup\Sigma_{i}$.

\item[(c)] If $\Gamma_{i-1}\cup\Sigma_{i}$ is invalid, then by Theorem \ref{cup}, $\Gamma_{i-1}$ and $\Sigma_{i}$ have no common edges and $\Sigma_{i}$ is invalid.
By condition (ii), $\Sigma_{i}$ contains an $E$-rung, say, $e_{i}$, which is different from $e_{i-1}$. Since $\Gamma_{i-1}$ and $\Sigma_{i}$ have no common edges, by the definition of $\Gamma_{i-1}$ we see that $e_{i}$ is not contained in $\Sigma_{1}\cup\cdots\cup\Sigma_{i-1}$. Set $\Gamma_{i} = \Gamma_{i-1}\cup(\Sigma_{i} - e_{i})$. By Lemmas \ref{union} and \ref{minus}, $\Gamma_{i}$ is valid.

\item[(d)] It can be verified that $\Gamma_{i}=\Sigma_{1}\cup\cdots\cup\Sigma_{i}$ or $\Gamma_{i}=(\Sigma_{1}\cup\cdots\cup\Sigma_{i}) - e_{i}$.

\end{itemize}
\end{enumerate}

When the algorithm above terminates, we obtain $\Gamma_{n}=\Sigma_{1}\cup\cdots\cup\Sigma_{n}$. That is, $\Ga_n$ is the union of the subgraphs in $\mathcal{F}$. Thus, by condition (i), we have $\Gamma=\Gamma_{n}$. Since $\Gamma_{n}$ is valid, it follows that $\Gamma$ admits a nowhere-zero $3$-flow.
\end{proof}

\section{Technical lemmas}
\label{sec:lem}

In this section we prove several technical lemmas which will be used in the proof of Theorems \ref{Cayley} and \ref{Cayley1}. Let us fix our notation first. As usual, for groups $K$ and $H$, we use $ K\times H$ and $K\rtimes H$ to denote the directed product of $K$ and $H$ and the semidirect product of $K$ by $H$, respectively. An \emph{involution} of a group is an element of order $2$. An involution is called a \emph{central involution} if it is commutable with every element of the group. Let $X$ and $Y$ be connection multisets (sets) of a group $G$. If $Y$ is a submultiset (subset) of $X$, then we call $Y$ a \emph{connection submultiset} (\emph{subset}) of $X$. It is well known that every connected component of $\Cay(G,X)$ is isomorphic to $\Cay(\langle X\rangle,X)$. In fact, for any $g\in G$, the connected component of $\Cay(G,X)$ containing $g$ is a \emph{translation} of $\Cay(\langle X\rangle,X)$ in the sense that it is obtained by left-multiplying each vertex of $\Cay(\langle X\rangle,X)$ by $g$ and retaining the adjacency relation. Let $N$ be a normal subgroup of $G$. If $X$ is a connection multiset of $G$ such that $N \cap X = \emptyset$, then $X/N := \{Nx: x \in X\}$ is a connection multiset of the quotient group $G/N$ and thus $\Cay(G/N, X/N)$ is well defined, where the multiplicity in $X/N$ of each $Nx \in X/N$ is defined as the sum of the multiplicities of all $y$ in $X$ with $Ny = Nx$. We call $\Cay(G/N, X/N)$ the \emph{quotient} of $\Cay(G,X)$ with respect to $N$. Note that some elements of $X/N$ may have multiplicities greater than $1$ even if all elements of $X$ have multiplicities $1$. That is, $\Cay(G/N, X/N)$ may have parallel edges even if $\Cay(G,X)$ is a simple Cayley graph.

\begin{lem}
\label{quotient}
{\rm(\cite[Proposition 4.1]{NS2009})}
Let $G$ be a group and $N$ a normal subgroup of $G$. Let $\Cay(G,X)$ be a Cayley graph on $G$ such that $N \cap X=\emptyset$. If $\Cay(G/N, X/N)$ admits a nowhere-zero $k$-flow, then so does $\Cay(G,X)$.
\end{lem}

The following lemma is the main result in \cite{PSS2005}.

\begin{lem}
\label{abelian}
{\rm \cite[Theorems 1.1 and 4.2]{PSS2005}}
Every Cayley graph of valency at least four on an abelian group admits a nowhere-zero $3$-flow.
\end{lem}

The next lemma is a collection of a few known results on supersolvable groups (see, for example, \cite{R1995}).

\begin{lem}
\label{supersolvable}
Let $G$ be a supersolvable group. Then the following hold:
\begin{enumerate}
 \item the elements of $G$ with odd orders form a characteristic subgroup of $G$;
  \item the derived group $G'$ is nilpotent;
  \item every subgroup of $G$ is supersolvable, and for every normal subgroup $N$ of $G$ the quotient group $G/N$ is supersolvable;
  \item every minimal normal subgroup of $G$ is a cyclic group of prime order.
\end{enumerate}
\end{lem}

\begin{lem}
\label{derived}
Let $G$ be a group. If the derived group $G'$ is of square-free order, then $G$ is supersolvable.
\end{lem}

\begin{proof}
Since $G/G'$ is abelian, it possesses a normal series $G'/G'=N_{0}/G'\leq N_{1}/G' \leq \cdots \leq N_{k}/G' =G/G'$ such that the quotient group $(N_{i}/G')/(N_{i-1}/G')$ is cyclic. Since $G'$ is of square-free order, every Sylow subgroup of $G'$ is cyclic. By \cite[10.1.10]{R1995}, $G'$ is a cyclic group or a metacyclic group. Therefore, $G'$ has a cyclic normal subgroup $G_{1}$ such that $G'/G_{1}$ is cyclic. Since $G'$ is of square-free order, the order of $G_{1}$ and the index of $G_{1}$ in $G'$ are coprime. It follows that $G_{1}$ is a characteristic subgroup of $G'$  and therefore a normal subgroup of $G$. Set $G_{i}=N_{i-2}G'$ for $2 \le i \le k+2$. Then we obtain a normal series $\{1\} = G_{0}\leq G_{1}\leq\cdots \leq G_{k+1}\leq G_{k+2}=G$ such that the quotient group $G_{i}/G_{i-1}$ is cyclic. Therefore, $G$ is supersolvable.
\end{proof}

\begin{lem}
\label{noncyclic}
Let $G$ be a supersolvable group with a noncyclic Sylow $2$-subgroup, and $N$ a normal subgroup of $G$ contained in $G'$. Then any Sylow $2$-subgroup of $G/N$ is noncyclic.
\end{lem}

\begin{proof}
Let $H$ be the set of elements of $G$ with odd orders.
By Lemma \ref{supersolvable} (i), $H$ is a normal subgroup of $G$. Let $Q$ be a Sylow $2$-subgroup of $G$. Then $G=QH$ and $Q\cap H=\{1\}$. Therefore, $G/H\cong Q$. It follows that $\Phi(Q)H/H$ is the Frattini subgroup of $G/H$, where $\Phi(Q)$ is the Frattini subgroup of $Q$. In particular, $\Phi(Q)H$ is a normal subgroup of $G$. Since $Q$ is a noncyclic $2$-group, $Q/\Phi(Q)$ is an elementary abelian $2$-group of rank at least $2$ and hence is noncyclic. Since $G/\Phi(Q)H=QH/\Phi(Q)H\cong Q/\Phi(Q)$,  $G/\Phi(Q)H$ is a noncyclic abelian $2$-group. Therefore, $G'$ is contained in $\Phi(Q)H$. Since $N$ is contained in $G'$, $N$ is contained in $\Phi(Q)H$. Thus, since $(G/N)/(\Phi(Q)H/N)\cong G/\Phi(Q)H$ and $G/\Phi(Q)H$ is a noncyclic $2$-group, any Sylow $2$-subgroup of $G/N$ is noncyclic.
\end{proof}

The following lemma should be known (perhaps as a folklore or an exercise) in the mathematical community. Since we are unable to identify a reference to it, we give its proof here for completeness.

\begin{lem}\label{transversal}
Let $G$ be a group, and let $C$ and $H$ be subgroups of $G$ such that $C\cap H=\{1\}$. Then there exists a left transversal of $H$ in $G$ that contains $C$ as a subset.
\end{lem}

\begin{proof}
Let $gH$ be an arbitrary left coset of $H$ in $G$. Then for any $g\in G$ we have $|C\cap gH|\leq1$, for otherwise we would have $1 \neq c^{-1}c'\in C\cap H$ for distinct $c, c' \in C\cap gH$, which contradicts the assumption that $C\cap H=\{1\}$. So we can take  a left transversal $\{g_1,g_2,\ldots,g_m,g_{m+1},\ldots,g_{n}\}$ of $H$ in $G$ such that $|C\cap g_{i}H|=1$ for $1\leq i\leq m$ and $|C\cap g_{i}H|=0$ for $m+1\leq i\leq n$, where $n$ is the index of $H$ in $G$ and $m$ is some integer between $1$ and $n$. Set $C\cap g_{i}H=\{c_i\}$ for $1\leq i\leq m$, and let $A=\{c_1,c_2,\ldots,c_m,g_{m+1},\ldots,g_{n}\}$. Then $A$ is a left transversal of $H$ in $G$ such that $C\subseteq A$.
\end{proof}


\begin{lem}
\label{normal}
Let $\Cay(G,X)$ be a cubic connected Cayley graph on a group $G$ of order at least four. If $G$ contains a normal cyclic subgroup generated by a prime order element of $X$, then $\Cay(G,X)$ is a closed ladder with at least two rungs.
\end{lem}

\begin{proof}
Denote $\Ga = \Cay(G,X)$. Since $\Ga$ is connected with valency three, we have $G=\langle X\rangle$ and $X$ is of cardinality $3$. Hence $X$ contains at least one involution. Let us first consider the case when $X$ contains at least one central involution of $G$. In this case, one of the following possibilities occurs:
\begin{itemize}
  \item $G=\langle y\rangle\times \langle z\rangle$, $X=\{y,z,z\}$, where both $y$ and $z$ are involutions, and $\Ga$ is isomorphic to $CL_{2}$;
  \item $G=\langle x,y\rangle$, $X=\{x,y,z\}$, where both $x$ and $y$ are involutions, $xy$ is of order $2n$ and $z=(xy)^{n}$ for some $n \ge 1$, and $\Ga$ is isomorphic to $M_{2n}$;
  \item $G=\langle x, y\rangle\times \langle z\rangle$,  $X=\{x,y,z\}$, where $x,y,z$ are all involutions and $xy$ is of order $n$ for some $n \ge 2$, and $\Ga$ is isomorphic to $CL_{2n}$;
  \item $G=\langle x\rangle\times \langle z\rangle$, $X=\{x,x^{-1},z\}$, where $x$ is of order $n>2$, $z$ is an involution, and $\Ga$ is isomorphic to $CL_{n}$;
  \item  $G=\langle x\rangle$, $X=\{x,x^{-1},z\}$, where $x$ is of order $2n$ and $z=x^{n}$ for some $n \geq 2$, and $\Ga$ is isomorphic to $M_{n}$.
\end{itemize}
In each possibility above, $\Ga$ is a closed ladder with at least two rungs, as desired.

Now we assume that $X$ contains no central involution of $G$. Then by our assumption there exists a normal cyclic subgroup $\langle x\rangle$ of $G$ with order an odd prime $p$ such that $x\in X$. It follows that $X=\{x,x^{-1}, z\}$, where $z$ is an involution. Therefore, $G=\langle x\rangle\rtimes \langle z\rangle$ is the dihedral group of order $2p$ and $\Ga$ is isomorphic to $CL_{p}$.
\end{proof}

Recall from the previous section that closed ladders are defined as certain Cayley graphs. Let $\Cay(G,X)$ be a closed ladder. We call an involution $z$ of $G$ the \emph{rung involution} for $\Cay(G,X)$ if $z \in X$ and for $g, h \in G$ there is a rung of $\Cay(G,X)$ joining $g$ and $h$ precisely when $g^{-1}h = z$.

\begin{lem}
\label{2close}
Let $\Cay(G,X)$ be a Cayley graph of valency at least five. Suppose that $X$ contains two inverse-closed submultisets $\{u_{1},u_{2}\}$, $\{v_{1},v_{2}\}$ and an involution $z$ such that $\Cay(\langle u_{1},u_{2},z\rangle,\{u_{1},u_{2}, z\})$ and $\Cay(\langle v_{1},v_{2},z\rangle,\{v_{1},v_{2}, z\})$ are both closed ladders with $z$ as their rung involution. Then $\Cay(G,X)$ admits a nowhere-zero $3$-flow.
\end{lem}

Note that in this lemma $\{u_{1},u_{2}\}$ and $\{v_{1},v_{2}\}$ are allowed to have one or two common elements and each of them can be a multiset (that is, we allow $u_1 = u_2$ for an involution $u_1$ and/or $v_1 = v_2$ for an involution $v_1$). However, $\{u_{1},u_{2}\}$ and $\{v_{1},v_{2}\}$ are treated as distinct submultisets of the multiset $X$.

\medskip
\begin{ProofL}\textbf{\ref{2close}.}~
Denote $\Gamma=\Cay(G,X)$. Since every regular graph of even valency admits a nowhere-zero $3$-flow, $\Gamma$ admits a nowhere-zero $3$-flow if $X$ is of even cardinality. Now assume that $X$ is of odd cardinality.  Set $U=\{u_{1},u_{2},z\}$, $V=\{v_{1},v_{2},z\}$, and $Y=\{u_{1},u_{2},v_{1},v_{2},z\}$.
Then $\Cay(G,Y)$ is a spanning subgraph of $\Gamma$ of valency five. Let $E$ be the set of edges of $\Cay(G,Y)$ joining $g$ and $gz$ for $g\in G$, and let
\begin{equation*}
 \mathcal{F}=\{\Sigma: \Sigma~\mbox{is a connected component of}~\Cay(G,U)~\mbox{or}~\Cay(G,V)\}.
\end{equation*}
Since both $U$ and $V$ are submultisets (or subsets) of $Y$, $\mathcal{F}$ is a family of subgraphs of $\Cay(G,Y)$. Since $\Cay(\langle U\rangle,U)$ and $\Cay(\langle V\rangle,V)$ are both closed ladders with $z$ as their rung involution, each member of $\mathcal{F}$ is a closed ladder with rungs in $E$. It is straightforward to verify that $\Sigma$, $E$ and $\mathcal{F}$ meet all conditions in Theorem \ref{cladder}. Therefore, by Theorem \ref{cladder}, $\Cay(G,Y)$ admits a nowhere-zero $3$-flow. Since $\Cay(G,Y)$ is a regular spanning subgraph of $\Ga$ with odd valency, it follows from Lemma \ref{regular} that $\Ga$ admits a nowhere-zero $3$-flow.
\qed
\end{ProofL}

\begin{lem}
\label{central}
{\rm (\cite[Theorem 3.3]{NS2009})}
Let $\Cay(G,X)$ be a Cayley graph of valency at least four such that $X$ contains a central involution. Then $\Cay(G,X)$ admits a nowhere-zero $3$-flow.
\end{lem}

\begin{lem}
\label{index}
Let $\Cay(G,X)$ be a connected simple Cayley graph of valency five. Suppose that $X=U\cup\{b,b^{-1},z\}$, where $U$ is an inverse-closed subset of $G$ consisting of two elements of even order, $z$ is an involution of $G$, and $b$ generates a normal subgroup $\langle b\rangle$ of $G$ with order greater than $2$. If $\langle U\rangle\cap\langle b\rangle=\{1\}$ and $C_{G}(\langle b\rangle)$ is of index at most $2$ in $G$, then $\Cay(G,X)$ admits a nowhere-zero $3$-flow.
\end{lem}

\begin{proof}
Denote $\Gamma=\Cay(G,X)$. Since $\langle b\rangle$ is a cyclic normal subgroup of $G$, every subgroup of $\langle b\rangle$ is normal in $G$. If $b$ is of even order and $z\in\langle b\rangle$, then $z$ is a central involution of $G$. By Lemma \ref{central}, $\Gamma$ admits a nowhere-zero $3$-flow. If $b$ is of even order and $z\notin\langle b\rangle$, then $\langle b,z\rangle$ is a semidirect product of $\langle b\rangle$ and $\langle z\rangle$, and $bz=zb^{r}$ for some odd integer $r$. Therefore, the Cayley graph $\Cay(\langle b,z\rangle,\{b,b^{-1},z\})$ is a cubic bipartite graph with bipartition $\{\langle b^{2}\rangle\cup\langle b^2\rangle bz, \langle b^{2}\rangle z\cup\langle b^2\rangle b\}$. Since every connected component of $\Cay(G,\{b,b^{-1},z\})$ is isomorphic to $\Cay(\langle b,z\rangle,\{b,b^{-1},z\})$, $\Cay(G,\{b,b^{-1},z\})$ is a spanning cubic bipartite subgraph of $\Gamma$. By Lemma \ref{regular}, $\Gamma$ admits a nowhere-zero $3$-flow.

In the sequel we assume that $b$ is of odd order, say, $2n+1$. Then $z\notin \langle b\rangle$ and $\langle b,z\rangle$ is a semidirect product of $\langle b\rangle$ and $\langle z\rangle$. Suppose that $2n+1$ is not a prime. Then $\langle b\rangle$ contains a nontrivial proper subgroup. Let $N$ be such a subgroup of $\langle b\rangle$. Since $\langle b\rangle$ is normal in $G$, $N$ is normal in $G$. Clearly, $U$ is a connection set of $G$ with cardinality $2$. Since $U$ consists of two elements of even order and $z$ is an involution, $\langle U\rangle$ is of even order and $X\cap N=\emptyset$. By Lemma \ref{quotient}, $\Gamma$ admits a nowhere-zero $3$-flow if $\Cay(G/N,X/N)$ does. It remains to deal with the case when $2n+1$ is a prime (that is, $\langle b\rangle$ has no nontrivial proper subgroups).

From now on we assume that $2n+1$ is a prime. As mentioned above, the order of $\langle U\rangle$ is even, which we denote by $2m$. Since $C_{G}(\langle b\rangle)$ is of index at most $2$ in $G$, we have $g^{-2}bg^{2}=b$ for any $g\in G$. It follows that $g^{-1}bg=b~\mbox{or}~b^{-1}$. In particular, since $z$ is an involution, $\langle b,z\rangle$ is either an abelian group or a dihedral group. Since $U$ is a connection set of $G$ consisting of two elements of even order and $\langle U\rangle$ is of order $2m$, $\Cay(\langle U\rangle,U)$ is isomorphic to the cycle $C_{2m}$ of length $2m$. We will construct two subgraphs $\Sigma$, $\Lambda$ of $\Gamma$ satisfying the conditions in Theorem \ref{cup} and thus complete the proof by invoking this result.

\smallskip
\textsf{Step 1.} Construction of $\Sigma$.
\smallskip

Let $A$ be a left transversal of $\langle U,b\rangle$ in $G$. Then each left coset of $\langle U,b\rangle$ in $G$ is of the form $a \langle U,b\rangle$, where $a\in A$, and the subgraphs of $\Cay(G, U\cup\{ b, b^{-1}\})$ induced by these left cosets are precisely the connected components of $\Cay(G, U\cup\{ b, b^{-1}\})$. Denote by $\Theta_a$ the component of $\Cay(G, U\cup\{ b, b^{-1}\})$ induced by $a \langle U,b\rangle$, where $a \in A$. Then $\Theta_a$ is a translation of $\Theta_{1} := \Cay(\langle U,b\rangle, U\cup\{ b, b^{-1}\})$.
Since $\langle b\rangle$ is a cyclic normal subgroup of $G$ of order $2n+1$ and $\langle U\rangle\cap\langle b\rangle=\{1\}$, we have $\langle U,b\rangle=\langle b\rangle\langle U\rangle=\cup_{j=0}^{2n}b^{j}\langle U\rangle$ and $b^{i}\langle U\rangle\cap b^{j}\langle U\rangle=\emptyset$ for $i\neq j$. Denote by $\Theta_1^{(0)}$ the induced subgraph of $\Theta_1$ on $\langle U\rangle$, and denote by $\Theta_1^{(i)}$ the induced subgraph of $\Theta_1$ on $b^{2i-1}\langle U\rangle\cup b^{2i}\langle U\rangle$ for $1 \le i \le n$. Then $\Theta_1^{(0)},\Theta_1^{(1)},\ldots,\Theta_1^{(n)}$ are pairwise vertex-disjoint. Note that $\Theta_1^{(0)}$ is the Cayley graph $\Cay(\langle U\rangle,U)$, which is isomorphic to the cycle $C_{2m}$. Note also that $\{b^{2i-1}h,b^{2i}h\}$ is an edge of $\Theta_1^{(i)}$ for every $h\in \langle U\rangle$ as $(b^{2i-1}h)^{-1}(b^{2i}h)=h^{-1}b^{-(2i-1)}b^{2i}h=h^{-1}bh
=b~\mbox{or}~b^{-1}$. (Recall that $g^{-1}bg=b~\mbox{or}~b^{-1}$ for any $g\in G$.)
It follows that $\Theta_1^{(i)}$ is isomorphic to the circular ladder $CL_{2m}$ with rung set $R_{1}^{(i)}:=\big\{\{b^{2i-1}h,b^{2i}h\}\mid h\in \langle U\rangle\big\}$ and rail edges of the form $\{g,gy\}$, where $g\in b^{2i-1}\langle U\rangle\cup b^{2i}\langle U\rangle$ and $y\in U$. Denote by $\Theta_a^{(i)}$ the translation of $\Theta_1^{(i)}$ by $a$ for $0 \le i \le n$ and $a\in A$. Then we obtain a family
\begin{equation*}
 \mathcal{F}=\big\{\Theta_a^{(i)}\mid a\in A, 0 \le i \le n \big\}
\end{equation*}
of vertex-disjoint graphs each of which is isomorphic to either $C_{2m}$ or $CL_{2m}$. Of course, $\mathcal{F}$ is a family of edge-disjoint generalized closed ladders. Let $R_{a}^{(i)}$ be the translation of $R_{1}^{(i)}$ by $a$, for $a\in A$.
Set
$$
\Sigma=\cup_{\Theta\in \mathcal{F}}\Theta
$$
and
$$
R_{a}=\cup_{i=1}^{n}R_{a}^{(i)},\; R=\cup_{a\in A}R_{a}.
$$
Then $\Sigma=\Cay(G,U)\cup R$ and the construction above ensures that the connected components of $\Sigma$ are the graphs $\Theta\in \mathcal{F}$. Noting that $b^{j}h=hb^{j}$ or $b^{j}h=hb^{-j}=hb^{2n+1-j}$ for $1 \le j \le 2n$ and $h\in\langle U\rangle$, we have $\{ab^{2i-1}h,ab^{2i}h\}=\{ahb^{2i-1},ahb^{2i}\}$ or $\{ab^{2i-1}h,ab^{2i}h\}=\{ahb^{2(n-i+1)},ahb^{2(n-i+1)-1}\}
=\{ahb^{2(n-i+1)-1},ahb^{2(n-i+1)}\}$ for $\{ab^{2i-1}h,ab^{2i}h\}\in R_{a}^{(i)}$.
It follows that
\begin{equation*}
R_{a}=\big\{\{ahb^{2i-1},ahb^{2i}\}\mid h\in \langle U\rangle, 1 \le i \le n\big\}
\end{equation*}
and hence
$$
R=\big\{\{gb^{2i-1},gb^{2i}\}\mid g\in A\langle U\rangle, 1 \le i \le n\big\}.
$$

\smallskip
\textsf{Step 2.} Construction of $\Lambda$.
\smallskip

Since $\langle b\rangle$ is normal in $G$, we have $\langle U,b\rangle=\langle U\rangle\langle b\rangle$. Since $A$ is a left transversal of $\langle U, b\rangle$ in $G$ and $\langle U\rangle\cap\langle b\rangle=\{1\}$, $A\langle U\rangle$ is a left transversal of $\langle b\rangle$ in $G$. Therefore, there exist a permutation $\mu$ on $A\langle U\rangle$ and a mapping $\lambda$ from $A\langle U\rangle$ to $\{0,1,\ldots,2n\}$ such that $gz=\mu(g)b^{\lambda(g)}$ for $g\in A\langle U\rangle$. Since $z\notin\langle b\rangle$, we have $\mu(g)\neq g$. So $\mu$ has no fixed elements. Noting that $\mu(g)z=gb^{-\lambda(g)}~\mbox{or}~gb^{\lambda(g)}$, we have $\mu^{2}(g)=g$. Thus, $\mu$ is an involution. Let $T$ be a subset of $A\langle U\rangle$ with maximum cardinality such that $T\cap \mu(T)=\emptyset$. Since $\mu$ is an involution on $A\langle U\rangle$ fixing  no elements, we have $A\langle U\rangle=T\cup \mu(T)$. For each $g\in T$, denote by $\Lambda_{g}$ the induced subgraph of $\Cay(G,\{b,b^{-1},z\})$ on $g\langle b\rangle\cup \mu(g)\langle b\rangle$. Then $\Lambda_{g}$ is a connected component of $\Cay(G,\{b,b^{-1},z\})$, and moreover $\cup_{g\in T}\Lambda_{g} = \Cay(G,\{b,b^{-1},z\})$. Furthermore, $\Lambda_{g}$ is a translation of $\Cay(\langle b,z\rangle,\{b,b^{-1},z\})$ and hence is isomorphic to the circular ladder $CL_{2n+1}$. If $\lambda(g)$ is odd, then we set
\begin{equation*}
  \Lambda'_{g}=
     \begin{cases}
    \Lambda_g-\bigcup_{i=1}^n \big\{\{gb^{2i-1},gb^{2i}\},\{\mu(g)b^{2i-1},\mu(g)b^{2i}\}\big\},
    & \mbox{if}~\lambda(g)=1 \\
    \Lambda_g -\big\{\{gb^{2n+2-\lambda(g)},gb^{2n+3-\lambda(g)}\},
    \{\mu(g)b,\mu(g)b^{2}\}\big\}, &
    \mbox{if}~\lambda(g)\neq1~\mbox{and}~z^{-1}bz=b, \\
     \Lambda_g -\big\{\{gb,gb^{2}\},\{\mu(g)b^{\lambda(g)-2},\mu(g)b^{\lambda(g)-1}\}\big\}, &
    \mbox{if}~\lambda(g)\neq1~\mbox{and}~z^{-1}bz=b^{-1}. \\
     \end{cases}
\end{equation*}
If $\lambda(g)$ is even, then we set
\begin{equation*}
  \Lambda'_{g}=
     \begin{cases}
     \Lambda_g-\bigcup_{i=1}^n \big\{\{gb^{2i-1},gb^{2i}\},\{\mu(g)b^{2i-1},\mu(g)b^{2i}\}\big\},
      & \mbox{if}~\lambda(g)=2n \\
    \Lambda_g - \big\{\{gb,gb^{2}\},\{\mu(g)b^{\lambda(g)+1},\mu(g)b^{\lambda(g)+2}\}\big\}, &
    \mbox{if}~\lambda(g)\neq2n~\mbox{and}~z^{-1}bz=b, \\
     \Lambda_g - \big\{\{gb^{2n-1},gb^{2n}\},\{\mu(g)b^{\lambda(g)+1},\mu(g)b^{\lambda(g)+2}\}\big\}, &
    \mbox{if}~\lambda(g)\neq2n~\mbox{and}~z^{-1}bz=b^{-1}. \\
     \end{cases}
\end{equation*}
It can be verified that if $\lambda(g)\in\{1,2n\}$ then $\Lambda'_{g}$ is a subdivision of a generalized M\"obius ladder with a unique rung, and if $\lambda(g)\notin\{1,2n\}$ then $\Lambda'_{g}$ is isomorphic to $L_{2n+1}$. (See Figure \ref{fig:lambda} for two spacial cases where $z^{-1}bz=b^{-1}$, $n=3$, and $\lambda(g)$ is $6$ or $5$.) It follows that $\Lambda'_{g}$ admits a nowhere-zero $3$-flow. Since $A\langle U\rangle=T\cup\mu(T)$ is a left transversal of $\langle b\rangle$ in $G$ and $Tz\langle b\rangle=\mu(T)\langle b\rangle$,
$T\cup Tz$ is a left transversal of $\langle b\rangle$ in $G$. Therefore, $T$ is a left transversal of $\langle b,z\rangle$ in $G$. It follows that $\Lambda'_{g}$ and $\Lambda'_{h}$ have no common vertices for every pair of distinct $g,h\in T$. Now we set
$$
\Lambda:=\cup_{g\in T}\Lambda'_{g}.
$$
Then $\Lambda$ is the union of $|T|$ edge-disjoint subgraphs of $\Ga$ each admitting a nowhere-zero $3$-flow. By Lemma \ref{union}, we know that $\Lambda$ admits a nowhere-zero $3$-flow.

\begin{figure}[h]
\centering
\begin{tikzpicture}
\tikzstyle{every node}=[draw,shape=circle,label distance=-0.5mm,inner sep=1pt];
\node (0) at (-5,6) [fill,label=above left:\small{$g$}] {};
\node (1) at (-5,5) [fill,label=left:\small{$gb$}] {};
\node (2) at (-5,4) [fill,label=left:\small{$gb^{2}$}] {};
\node (3) at (-5,3) [fill,label=left:\small{$gb^{3}$}] {};
\node (4) at (-5,2) [fill,label=left:\small{$gb^{4}$}] {};
\node (5) at (-5,1) [fill,label=left:\small{$gb^{5}$}] {};
\node (6) at (-5,0) [fill,label=below left:\small{$gb^{6}$}] {};
\node (10) at (-3.75,6) [fill,label=above right:\small{$\mu(g)b^{6}$}] {};
\node (11) at (-3.75,5) [fill,label=right:\small{$\mu(g)b^{5}$}] {};
\node (12) at (-3.75,4) [fill,label=right:\small{$\mu(g)b^{4}$}] {};
\node (13) at (-3.75,3) [fill,label=right:\small{$\mu(g)b^{3}$}] {};
\node (14) at (-3.75,2) [fill,label=right:\small{$\mu(g)b^{2}$}] {};
\node (15) at (-3.75,1) [fill,label=right:\small{$\mu(g)b$}] {};
\node (16) at (-3.75,0) [fill,label=below right:\small{$\mu(g)$}] {};
\draw
(0) -- (1)(2)--(3) (4) -- (5)  (6)..controls (-7.5,0) and (-7.5,6)..(0)
(0) -- (10) (1) -- (11) (2) -- (12) (3) -- (13) (4) -- (14) (5) -- (15) (6) -- (16)
(11) -- (12)(13) -- (14)(15) -- (16)..controls (-1.25,0) and (-1.25,6)..(10);
\end{tikzpicture}
\begin{tikzpicture}
\tikzstyle{every node}=[draw,shape=circle,label distance=-0.5mm,inner sep=1pt];
\node (0) at (3.75,6) [fill,label=above left:\small{$g$}] {};
\node (1) at (3.75,5) [fill,label=left:\small{$gb$}] {};
\node (2) at (3.75,4) [fill,label=left:\small{$gb^{2}$}] {};
\node (3) at (3.75,3) [fill,label=left:\small{$gb^{3}$}] {};
\node (4) at (3.75,2) [fill,label=left:\small{$gb^{4}$}] {};
\node (5) at (3.75,1) [fill,label=left:\small{$gb^{5}$}] {};
\node (6) at (3.75,0) [fill,label=below left:\small{$gb^{6}$}] {};
\node (10) at (5,6) [fill,label=above right:\small{$\mu(g)b^{5}$}] {};
\node (11) at (5,5) [fill,label=right:\small{$\mu(g)b^{4}$}] {};
\node (12) at (5,4) [fill,label=right:\small{$\mu(g)b^{3}$}] {};
\node (13) at (5,3) [fill,label=right:\small{$\mu(g)b^{2}$}] {};
\node (14) at (5,2) [fill,label=right:\small{$\mu(g)b$}] {};
\node (15) at (5,1) [fill,label=right:\small{$\mu(g)$}] {};
\node (16) at (5,0) [fill,label=below right:\small{$\mu(g)b^{6}$}] {};
\draw
(0)--(1)(2) -- (3) -- (4) -- (5)-- (6)..controls (1.25,0) and (1.25,6)..(0)
(0) -- (10) (1) -- (11) (2) -- (12) (3) -- (13) (4) -- (14) (5) -- (15) (6) -- (16)
(10)--(11) (12) -- (13) -- (14) -- (15)-- (16)..controls (7.5,-0) and (7.5,6)..(10);
\end{tikzpicture}
\caption{The subgraphs $\Lambda'_{g}$, $g \in T$ of $\Ga$ in the spacial cases when $z^{-1}bz=b^{-1}$, $n=3$, and $\lambda(g)$ is $6$ or $5$.
\label{fig:lambda}}
\end{figure}
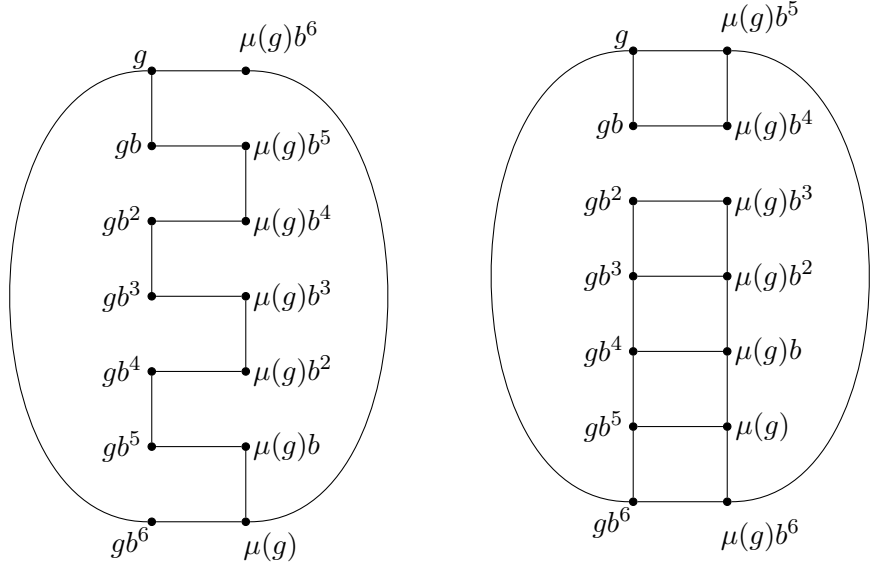

\smallskip
\textsf{Step 3.} Completing the proof.
\smallskip

In summary, we have constructed two subgraphs $\Lambda$, $\Sigma$ of $\Gamma$ satisfying the following conditions:
\begin{itemize}
  \item $\Lambda$ admits a nowhere-zero $3$-flow;
  \item $\Sigma=\cup_{\Theta\in \mathcal{F}}\Theta$, where $\mathcal{F}$ is a family of edge-disjoint generalized closed ladders, each of which is isomorphic to $C_{2m}$ or $CL_{2m}$ and is a connected component of $\Sigma$;
  \item for each $\Theta\in \mathcal{F}$, every common edge of $\Theta$ and $\Lambda$, if it exists, is a rung of $\Theta$.
\end{itemize}
Since by Lemma \ref{ladder}, $CL_{2m}$ admits a nowhere-zero $3$-flow, we see that every $\Theta \in \mathcal{F}$ admits a nowhere-zero $3$-flow. Therefore, $\Lambda$ and $\Sigma$ satisfy the conditions in Theorem \ref{cup}. Thus, by Theorem \ref{cup}, $\Sigma\cup\Lambda$ admits  a nowhere-zero $3$-flow.
Recall that $\Sigma=\Cay(G,U)\cup R$, where $R=\big\{\{gb^{2i-1},gb^{2i}\}\mid g\in A\langle U\rangle, 1 \le i \le n\big\}$. On the other hand, we see from the construction of $\Lambda$ that $\Cay(G,U\cup\{b,b,^{-1}z\})-R$ is a subgraph of $\Lambda$. Therefore, $\Gamma=\Sigma\cup\Lambda$ and we conclude that $\Gamma$ admits a nowhere-zero $3$-flow.
\end{proof}

\section{Proof of Theorems \ref{Cayley} and \ref{Cayley1}}
\label{sec:proof}

Now we are ready to prove our main results. As it turns out, it is convenient for us to include the case of Cayley graphs on nilpotent groups in our proof. However, this does not mean that our proof gives an independent proof of \cite[Theorem 4.3]{NS2009} as it relies on \cite[Theorem 3.3]{NS2009} (Lemma \ref{central}).

\medskip
\Proofs \textbf{\ref{Cayley} and \ref{Cayley1}.}~
Let $G$ be a group satisfying one of the following conditions:
\begin{enumerate}
   \item $G$ is nilpotent;
  \item $G$ is supersolvable with a noncyclic Sylow $2$-subgroup;
  \item the derived subgroup $G'$ of $G$ is of square-free order.
\end{enumerate}
Since every finite nilpotent group is supersolvable (see \cite[5.4.6]{R1995}) and any group satisfying (iii) is supersolvable (Lemma \ref{derived}), we see that $G$ is supersolvable. Clearly, if $G$ satisfies (iii), then every subgroup of $G$ satisfies (iii). Thus, to prove Theorem \ref{Cayley1}, it suffices to consider connected Cayley graphs on $G$ satisfying (iii), for otherwise we consider their connected components. Also, the Cayley graphs in Theorem \ref{Cayley} are connected by assumption.

By induction on the order $|G|$ of $G$, we will prove the statement that any connected Cayley graph on $G$ with valency at least four admits a nowhere-zero $3$-flow. By Lemma \ref{abelian}, this is true when $|G|<6$ as $G$ must be abelian in this case. Now assume that $|G| \geq 6$ and the statement is true for every supersolvable group which satisfies one of (i), (ii) and (iii) and is of order less than $|G|$. Let $\Ga = \Cay(G, X)$ be any connected Cayley graph on $G$ with valency at least four. We aim to show that $\Gamma$ admits a nowhere-zero $3$-flow. Once this is achieved the proof is complete by mathematical induction.

If $\Gamma$ is of even valency, then it admits a nowhere-zero $3$-flow and we are done. In what follows we assume that $\Gamma$ is of odd valency; that is, $X$ is a connection multiset of $G$ with cardinality an odd integer no less than $5$. By Lemma \ref{abelian}, if $G$ is abelian, then $\Gamma$ admits a nowhere-zero $3$-flow. Henceforth we further assume that $G$ is nonabelian so that $G'$ is nontrivial. Let $N$ be a minimal normal subgroup of $G$ contained in $G'$. By Lemma \ref{supersolvable} (iv), $N$ is a cyclic group of prime order, say, $p$. Denote by $Y$ (respectively, $W$) the submultiset of $X$ which consist of those elements of $X$ contained in $G \setminus N$ (respectively, $N$). Of course, $Y$ and $W$ form a partition of $X$, and so their cardinalities have different parity as the cardinality of $X$ is odd. Moreover, both $Y$ and $W$ are connection multisets of $G$.

\smallskip
\textsf{Case 1.} $Y$ is of even cardinality.

\smallskip
In this case $W$ is of odd cardinality. So $W$ contains at least one involution, say, $z$, as it is a connection multiset. Thus $p=2$ and $N=\langle z\rangle$. Since $N$ is normal in $G$, $z$ is a central involution of $G$. Therefore, by Lemma \ref{central}, $\Gamma$ admits a nowhere-zero $3$-flow.

\smallskip
\textsf{Case 2.} $Y$ is of odd cardinality no less than $5$.
\smallskip

By Lemma \ref{supersolvable} (iii),  $G/N$ is a supersolvable group.
Since $G=\langle X\rangle$ and $Y$ contains all elements of $X$ which are not in $N$, we have $G/N=\langle Y/N\rangle$. Hence the quotient $\Cay(G/N,Y/N)$ of $\Cay(G,Y)$ is a connected Cayley graph on $G/N$. Since $Y$ is of odd cardinality no less than $5$, $\Cay(G,Y)$ is of odd valency no less than $5$. Noting that $\Cay(G/N,Y/N)$ and $\Cay(G,Y)$ have the same valency, we see that $\Cay(G/N,Y/N)$ is of odd valency no less than $5$. Since every quotient group of a nilpotent group is nilpotent, condition (i) holds for $G/N$ if it holds for $G$. By Lemma \ref{noncyclic}, condition (ii) holds for $G/N$ if it holds for $G$. Since $(G/N)/(G'N/N)$ is isomorphic to $G/G'N$ which is abelian, $(G/N)'$ is contained in $G'N/N$. Therefore, condition (iii) holds for $G/N$ if it holds for $G$. Since $G$ satisfies (i), (ii) or (iii), we obtain that $G/N$ satisfies (i), (ii) or (iii), respectively. Thus, by the induction hypothesis, $\Cay(G/N,Y/N)$ admits a nowhere-zero $3$-flow. Hence, by Lemma \ref{quotient}, $\Cay(G,Y)$ admits a nowhere-zero $3$-flow. Since $\Cay(G,Y)$ is a spanning subgraph of $\Gamma$ with odd valency, by Lemma \ref{regular}, we conclude that $\Gamma$ admits a nowhere-zero $3$-flow.

\smallskip
\textsf{Case 3.} $Y$ is of cardinality 1.
\smallskip

In this case $Y=\{y\}$ for some involution $y$ of $G$. Since $N$ is a normal subgroup of $G$ with order $p$ and $G=\langle X\rangle=\langle y,N\rangle=\langle N\rangle\rtimes\langle y\rangle$, we have $|G|=2p$. Recall that $|G|\geq6$ and $p$ is odd. Since $X$ is a connection set or connection multiset of odd cardinality no less than $5$, there exist $a,b\in N$ such that $Z:=\{y,a,a^{-1},b,b^{-1}\}$ is a subset or submultiset of $X$. By Lemma \ref{normal}, $\Cay(G,\{y,a,a^{-1}\})$ and $\Cay(G,\{y,b,b^{-1}\})$ are both closed ladders with $y$ as their rung involution. Therefore, by Lemma \ref{2close}, $\Gamma$ admits a nowhere-zero $3$-flow.

\smallskip
\textsf{Case 4.} $Y$ is of cardinality $3$ and $G$ has a minimal normal subgroup $L$ other than $N$ which is contained in $G'$.
\smallskip

Note that $L\cap N=\{1\}$ as $L$ and $N$ are distinct minimal normal subgroups of $G$. If $L\cap X=\emptyset$, then the problem is boiled down to Case 2 by replacing $N$ by $L$. Now we assume $L\cap X\neq\emptyset$. Since $Y$ includes all elements of $X$ which are not contained in $N$, we have $L\cap Y\neq\emptyset$. By Lemma \ref{normal}, $\Cay(\langle Y\rangle,Y)$ is a closed ladder. Set $X=\{x,y,z,a,b,\ldots\}$, where $x,y,z\in Y$, $z$ is the rung involution for $\Cay(\langle Y\rangle,Y)$, and $a, b \in N$ with $\{a, b\}$ inverse-closed. By the proof of Lemma \ref{normal}, $\Cay(\langle z,a,b\rangle,\{z,a,b\})$ is also a closed ladder with $z$ as its rung involution. Therefore, by Lemma \ref{2close}, $\Gamma$ admits a nowhere-zero $3$-flow.

\smallskip
\textsf{Case 5.} $Y$ is of cardinality $3$ and $N$ is the only minimal normal subgroup of $G$ contained in $G'$.
\smallskip

In this case $Y$ is of the form $Y=\{x,y,z\}$, where $z$ is an involution and $\{x,y\}$ is inverse-closed. If $X$ is of odd cardinality no less than $7$, then $X$ contains a submultiset $\{z,a_{1},a_{2},b_{1},b_{2}\}$ such that $a_{1},a_{2}, b_{1},b_{2} \in N$ and both $\{a_{1},a_{2}\}$ and $\{b_{1},b_{2}\}$ are inverse-closed. By the proof of Lemma \ref{normal}, $\Cay(\langle a_{1},a_{2}\rangle,\{a_{1},a_{2},z\})$ and $\Cay(\langle b_{1},b_{2}\rangle,\{b_{1},b_{2},z\})$ are both closed ladders with $z$ as their rung involution. Thus, by Lemma \ref{2close}, $\Gamma$ admits a nowhere-zero $3$-flow when $X$ is of odd cardinality no less than $7$.

Now we assume that $X$ is of cardinality $5$. Since $G$ is supersolvable, by Lemma \ref{supersolvable} (ii), $G'$ is nilpotent. It follows that every Sylow subgroup of $G'$ is characteristic in $G'$ and therefore normal in $G$. Since $N$ is the only minimal normal subgroup of $G$ contained in $G'$, $G'$ is a $p$-group. If $p=2$, then $X$ contains a central involution of $G$ and so $\Gamma$ admits a nowhere-zero $3$-flow by Lemma \ref{central}. Now we assume that $p>2$. Then $X=\{x,y,z,b,b^{-1}\}$ for some $b\in N$. By the N/C Lemma (\cite[1.6.13]{R1995}), the centralizer $C_{G}(N)$ of $N$ in $G$ is normal in $G$ and $G/C_{G}(N)$ is isomorphic to a subgroup of $\Aut(N)$. It is well known (see \cite[1.5.5]{R1995}) that the automorphism group of a cyclic group of order $p$ is a cyclic group of order $p-1$. Therefore, $G/C_{G}(N)$ is a cyclic group of order dividing $p-1$. Since $G'$ is a $p$-group and $N$ is a normal subgroup of $G'$ with order $p$, we obtain that $G'$ is contained in $C_{G}(N)$.

\smallskip
\textsf{Subcase 5.1.} Both $x$ and $y$ are involutions.
\smallskip

Then $x,y$ and $z$ are all involutions. Since $G'$ is a $p$-group, it follows that $\{x,y,z\}\cap G'=\emptyset$. Since $G=\langle x,y,z,b,b^{-1}\rangle$ (as $\Ga$ is connected) and $b\in G'$, we have  $G/G'=\langle xG',yG',zG'\rangle$. Noting that $xG',yG'$ and $zG'$ are all involutions, $G/G'$ is an elementary abelian $2$-group of order at most $8$. If $\Cay(G,\{x,y,z\})$ is bipartite, then $\Gamma$ admits a nowhere-zero $3$-flow by Lemma \ref{regular}. Now we assume that $\Cay(G,\{x,y,z\})$ is not bipartite. Then $\Cay(G,\{x,y,z\})$ is a simple graph. Since $X=\{x,y,z,b,b^{-1}\}$ and $b\notin \{x,y,z\}$, $\Gamma$ is a simple graph. Since $\Cay(G,\{x,y,z\})$ is not bipartite, the quotient $\Cay(G/G',\{xG',yG',zG'\})$ of $\Cay(G,\{x,y,z\})$ with respect to $G'$ is not a bipartite graph. Therefore, by the proof of Lemma \ref{normal}, we have $G/G'=\langle xG'\rangle\times\langle yG'\rangle$, $zG'=xyG'$, and $\Cay(G/G',\{xG',yG',zG'\})$ is isomorphic to the M\"obius ladder $M_{2}$. Recalling that $G/C_{G}(N)$ is a cyclic group of order dividing $p-1$ and $G'$ is contained in $C_{G}(N)$, we have $\{x,y,z\}\cap C_{G}(N)\neq\emptyset$. Without loss of generality, we may assume that $y\in C_{G}(N)$. Then $G=C_{G}(N)$ or $G/C_{G}(N)=\langle xC_{G}(N)\rangle$. In particular,  the index of $C_{G}(N)$ in $G$ is at most $2$. Set $h=xy$ and $U=\{x,y\}$. Then $\langle U\rangle=\langle h\rangle\rtimes\langle y\rangle$ is a dihedral group and thus is of even order. Note that $U$ is inverse-closed. Since $h^{y}=h^{-1}$ and $b^{y}=b\neq b^{-1}$, we have $b\notin \langle x,y\rangle$. Therefore, $\langle x,y,b\rangle=\langle b\rangle\rtimes\langle U\rangle$ and it follows that $\langle b\rangle\cap\langle U\rangle=\{1\}$. Thus, by Lemma \ref{index}, $\Gamma$ admits a nowhere-zero $3$-flow.

\smallskip
\textsf{Subcase 5.2.} $x=y^{-1}$ and $y$ is of order greater than $2$.
\smallskip

In this case we have $G=\langle y,b,z\rangle$ as $\Ga$ is connected. Since $b\in G'$, it follows that $G/G'=\langle yG',zG'\rangle$. Since $G'$ is a $p$-group and $p$ is an odd prime, any Sylow $2$-subgroup of $G$ is isomorphic to the Sylow $2$-subgroup of $G/G'$.

If $G$ has a noncyclic Sylow $2$-subgroup, then the Sylow $2$-subgroup of $G/G'$ is not cyclic. Since $zG'$ is an involution and $G/G'$ is abelian, $G/G'=\langle yG',zG'\rangle=\langle zG'\rangle\times\langle yG'\rangle$ and $yG'$ is of even order. It follows that $\Cay(G,\{y,y^{-1},z\})$ is a spanning cubic bipartite subgraph of $\Gamma$ with bipartition $\{\langle y^{2}\rangle G'\cup \langle y^{2}\rangle yz G', \langle y^{2}\rangle y G'\cup \langle y^{2}\rangle z G'\}$. By Lemma \ref{regular}, $\Gamma$ admits a nowhere-zero $3$-flow.

If $G$ is a nilpotent group with a cyclic Sylow $2$-subgroup, then $G$ has a unique involution which is a central involution. It follows that $z$ is a central involution of $G$. By Lemma \ref{central}, $\Gamma$ admits a nowhere-zero $3$-flow.

Henceforth we assume that any Sylow $2$-subgroup of $G$ is cyclic and $G$ is not a nilpotent group. By our assumption, $G'$ is of square-free order. On the other hand, $G'$ is a $p$-group as proved earlier. Hence $G'$ is of order $p$. Since $b\in G'$, it follows that $G'=\langle b\rangle=N$. Since any Sylow $2$-subgroup of $G$ is cyclic, the Sylow $2$-subgroup of $G/G'$ is cyclic. Write $c:=yz$. Since $G/G'=\langle yG',zG'\rangle$ is abelian and $zG'$ is an involution, we have $G/G'=\langle cG'\rangle$ and hence $G=\langle c\rangle\langle b\rangle$. Since $G$ is not a cyclic group, we have $b\notin\langle c\rangle$ and therefore
$\langle c\rangle\cap\langle b\rangle=\{1\}$. Note that $c$ is of even order, say, $2m$. Then $G/G'$ is of order $2m$ and it follows that $G$ is of order $2mp$. Therefore, $\langle c\rangle$ is a maximal subgroup of $G$ of index $p$.
If $y\in \langle c\rangle$, then $\langle c\rangle=\langle y,z\rangle$. By Lemma \ref{normal}, $\Cay(\langle y,z\rangle,\{y,y^{-1},z\})$ and $\Cay(\langle b,z\rangle,\{b,b^{-1},z\})$ are both circular ladders with $z$ as their rung involution. So, by Lemma \ref{2close}, $\Gamma$ admits a nowhere-zero $3$-flow. In what follows, we assume $y\notin \langle c\rangle$. Then $G=\langle y,c\rangle$.
It follows that $L:=\langle y\rangle\cap\langle c\rangle$ is a normal subgroup of $G$ and $y,z\notin L$. Since $b\notin\langle c\rangle$, we have $b\notin L$. Therefore, $X\cap L=\emptyset$. Clearly, $G/L$ is a supersolvable group satisfying condition (iii). If $L$ is nontrivial, then by the induction hypothesis, the quotient $\Cay(G/L,X/L)$ of $\Gamma$ admits a nowhere-zero $3$-flow and hence, by Lemma \ref{quotient}, $\Gamma$ admits a nowhere-zero $3$-flow. Now we assume that $L$ is trivial, that is, $\langle y\rangle\cap\langle c\rangle=1$. By Lemma \ref{transversal}, there exists a left transversal $A$ of $\langle y\rangle$ in $G$ such that $\langle c\rangle\subseteq A$. Since the order of $c$ is $2m$, we may write $A=\{a_{0},a_{1},\ldots,a_{2m-1},a_{2m},a_{2m+1},\ldots,a_{n-1}\}$, where $a_{i}=c^{i}$ for $0\leq i\leq 2m-1$. Since $zG'$ is an involution of $G/G'$, we have $zG'=c^{m}G'$.
Therefore, $\langle c^{m},b\rangle=\langle b, z\rangle$.
Since $G=\langle c\rangle\langle b\rangle$ and $\langle c\rangle\cap\langle b\rangle=\{1\}$, $\langle c\rangle$ is a left transversal of $\langle b\rangle$ in $G$ and it follows that $\{1,c^{-1},\ldots,c^{1-m}\}$ is a left transversal of $\langle z,b\rangle$ in $G$. Hence $\{a_{0}^{-1},a_{1}^{-1},\ldots,a_{m-1}^{-1}\}$ is a left transversal of $\langle z,b\rangle$ in $G$. It follows that every connected component of $\Cay(G,\{b,b^{-1},z\})$ is a translation $\Theta_{i}$ of $\Cay(\langle b,z\rangle,\{b,b^{-1},z\})$ by $a_{i}^{-1}$ for some $0\leq i\leq m-1$. Since
$\{a_{0},a_{1},\ldots,a_{n-1}\}$ is a left transversal of $\langle y\rangle$ in $G$, every connected component of $\Cay(G,\{y,y^{-1}\})$ is a translation $\Lambda_{j}$ of $\Cay(\langle y\rangle,\{y,y^{-1}\})$ by $a_{i}$ for some $0\leq i\leq n-1$. Note that each $\Lambda_{j}$ is a cycle. Note also that $\Theta_{0}\cong\Cay(\langle b,z\rangle, \{b,b^{-1},z\})$ and $\Theta_{i}$ is isomorphic to $\Theta_{0}$ for any $1\leq i\leq m-1$.
By Lemma \ref{normal}, $\Cay(\langle b,z\rangle, \{b,b^{-1},z\})$ is a circular ladder with $z$ as its rung involution. Therefore, $\Theta_{i}$ is a circular ladder with rungs of the form $\{a_{i}^{-1}b^{j},a_{i}^{-1}zb^{-j}\}$, where $0\leq j\leq p-1$ for any $0\leq i\leq m-1$. Denote by $\Delta$ the cycle in $\Gamma$ with vertices $1,y,c,cy,\ldots,c^{2m-1},c^{2m-1}y$ successively around the cycle. Then the following statements hold:

\smallskip
\textsf{Claim 1.} for $0\leq i\leq m-1$, $\{a_{i}^{-1}z,a_{i}^{-1}\}$ ($=\{c^{2m-i-1}y,c^{2m-i}\}$) is the unique common edge of $\Delta$ and $\Theta_{i}$, and moreover this edge is a rung of $\Theta_{i}$;

\smallskip
\textsf{Claim 2.} for $0\leq j\leq 2m-1$, $\{a_{j},a_{j}y\}$ is the unique common edge of $\Delta$ and $\Lambda_{j}$, and for $2m\leq j\leq n-1$, $\Delta$ and $\Lambda_{j}$ have no common edges.
\smallskip

Set $\Sigma=\Delta\cup\Cay(G,\{b,b^{-1},z\})$. Since $\Delta$ is a cycle, it admits a nowhere-zero $3$-flow. Note that $\Cay(G,\{b,b^{-1},z\})$ is the edge-disjoint union of the circular ladders $\Theta_{0},\Theta_{1},\ldots,\Theta_{m-1}$. Thus, by Claim 1 and Theorem \ref{cup}, $\Sigma$ admits a nowhere-zero $3$-flow. Clearly, $\Cay(G,\{y,y^{-1}\})$ is the edge-disjoint union of cycles $\Lambda_{0},\Lambda_{1},\ldots,\Lambda_{n-1}$. So $\Cay(G,\{b,b^{-1},z\})$ has no common edges with each $\Lambda_{j}$ for $0\leq j\leq n-1$ as it has no common edges with $\Cay(G,\{y,y^{-1}\})$. Thus, by Claim 2, $\Sigma$ has at most one common edge with each $\Lambda_{j}$ for $0\leq j\leq n-1$. Therefore, by Lemma \ref{union}, $\Sigma\cup\Cay(G,\{y,y^{-1}\})$ admits a nowhere-zero $3$-flow. However, this graph is exactly $\Ga$ as
\begin{equation*}
\Gamma=\Delta\cup\Gamma=\Delta\cup\big(\Cay(G,\{b,b^{-1},z\})\cup\Cay(G,\{y,y^{-1}\})\big)
=\Sigma\cup\Cay(G,\{y,y^{-1}\}).
\end{equation*}
Therefore, $\Gamma$ admits a nowhere-zero $3$-flow. This completes the proof by mathematical induction.
\qed

\noindent {\textbf{Acknowledgements}}
\medskip

The first author was supported by the Chinese Scholarship Council (No.201808505156), the Basic Research and Frontier Exploration Project of Chongqing (cstc2018jcyjAX0010) and the National Natural Science Foundation of China (No.11671276). The second author was supported by the Research Grant Support Scheme of The University of Melbourne.

{\small

\end{document}